\documentclass[a4paper,12pt,leqno]{amsart}
\usepackage{amsmath,amsthm,amssymb,latexsym,epsfig,graphicx,subfigure, color}
\usepackage{enumerate}
\usepackage{color}

\renewcommand{\epsilon}{\varepsilon}

\DeclareMathOperator{\dimH}{dim_H}
\DeclareMathOperator{\diam}{diam}
\DeclareMathOperator{\spt}{spt}

\newcommand{\J}{{\bf J}}
\newcommand{\bi}{{\bf i}}
\newcommand{\bk}{{\bf k}}
\newcommand{\bj}{{\bf j}}

\newtheorem{theorem}{Theorem}[section]
\newtheorem{lemma}[theorem]{Lemma}

\newtheorem{proposition}[theorem]{Proposition}

\theoremstyle{definition}

\theoremstyle{remark}
\newtheorem{remark}[theorem]{Remark}

\numberwithin{equation}{section}

\begin{document}

\title[Affine random covering sets in torus]
{Hausdorff dimension of affine random covering sets in torus}

\author[J\"arvenp\"a\"a]{Esa J\"arvenp\"a\"a$^1$}

\author[J\"arvenp\"a\"a]{Maarit J\"arvenp\"a\"a$^1$}

\author[Koivusalo]{Henna Koivusalo$^1$}

\author[Li]{Bing Li$^{2,1}$}

\author[Suomala]{Ville Suomala$^1$}

\address{$^1$Department of Mathematical Sciences, P.O. Box 3000, 90014
University of Oulu, Finland}
\address{$^2$Department of Mathematics, South China University of Technology,
Guangzhou, 510641, P.R. China}
\email{esa.jarvenpaa@oulu.fi, maarit.jarvenpaa@oulu.fi, henna.koivusalo@oulu.fi,
 libing0826@gmail.com, ville.suomala@oulu.fi}

\date{\today}
\thanks{We acknowledge the support of Academy of Finland, the Centre of Excellence in Analysis and Dynamics
Research. HK acknowledges the support of Jenny and Antti Wihuri Foundation.
We thank Ai-Hua Fan and the referee for useful comments.}

\begin{abstract}
We calculate the almost sure Hausdorff dimension of the random covering set
$\limsup_{n\to\infty}(g_n + \xi_n)$ in $d$-dimensional torus $\mathbb T^d$,
where the sets $g_n\subset\mathbb T^d$ are parallelepipeds, or more generally,
linear images of a set with nonempty interior, and $\xi_n\in\mathbb T^d$ are
independent and uniformly distributed random points. The dimension formula,
derived from the singular values of the linear mappings, holds provided that the
sequences of the singular values are decreasing.
\end{abstract}

\maketitle

\section{Introduction}

Given a sequence of positive numbers $(l_n)$ and a sequence of independent
random variables $(\xi_n)$ uniformly distributed on the circle $\mathbb T^1=\mathbb R/\mathbb Z$,
define {\it the random covering set} $E$ as follows:
\[
E=\{x\in \mathbb T^1\mid x\in [\xi_n,\xi_n + l_n] \text{ for infinitely many } n \}
=\limsup_{n\to\infty}\left[\xi_n,\xi_n + l_n\right].
\]
Denoting the Lebesgue measure by $\mathcal L$ and using the Borel-Cantelli lemma
and Fubini's theorem,
it follows that, almost surely, the following dichotomy holds:
\begin{equation}\label{01 for E}
\mathcal L(E)=\begin{cases}
	0, \text{ when }\sum_{n=1}^\infty l_n<\infty\\
	1, \text{ when }\sum_{n=1}^\infty l_n=\infty
	\end{cases},
\end{equation}
that is, almost all or almost no points of the circle
are covered, depending on whether or not the series of the lengths of the
covering intervals diverges.

The case of full Lebesgue measure has been extensively studied. It was
a long-standing problem to find conditions on $(l_n)$ guaranteeing that
the whole circle is covered almost surely, that is,
\begin{equation}\label{Dvoretzky}
P(E=\mathbb T^1)=1.
\end{equation}
This problem, known in literature as the Dvoretzky covering problem,
was first posed by Dvoretzky ~\cite{Dvoretzky56} in 1956. After substantial
contribution of many, including Kahane ~\cite{Kahane59},
Erd{\H{o}}s ~\cite{Erdos61}, Billard ~\cite{Billard65} and
Mandelbrot ~\cite{Mandelbrot72b},
the full answer was given by Shepp ~\cite{Shepp72} in 1972. He proved that
(\ref{Dvoretzky}) holds if and only if
\[
\sum_{n=1}^\infty \frac {1}{n^2} \exp(l_1 + \dots + l_n)=\infty,
\]
where the lengths $(l_n)$ are in decreasing order. After this, a natural
problem, raised by Carleson (private communication to Kahane), is to describe
the growth of the covering number of a given point $x\in \mathbb T^1$, that is, to
study the asymptotic behaviour of the sums
\[
C_N(x)=\sum_{n=1}^N \chi_{[\xi_n, \xi_n+l_n]}(x),
\]
where $\chi_A$ is the characteristic function of a set $A$. Obviously, the
expectation $\mathbb{E}(C_N(x))=\sum_{n=1}^Nl_n$. In the case
$l_n=\frac{\gamma}{n}$ with $\gamma>1$, Fan and Kahane \cite{FanKahane93}
proved that almost surely the order of the covering number $C_N(x)$ is $\log N$
for every $x\in \mathbb T^1$, meaning that for sufficiently large $N$
\[
A_{\gamma}\log N\le\min_{x\in \mathbb T}^1C_N(x)\le\max_{x\in \mathbb T^1}C_N(x)\le B_{\gamma}\log N
\]
with positive and finite constants $A_{\gamma}$ and $B_{\gamma}$. Furthermore, Fan
\cite{Fan} verified that the set
\[
F_{\beta}=\{x\in \mathbb T^1\mid \lim_{N\to\infty}\frac{C_N(x)}{\sum_{n=1}^Nl_n}=\beta \}
\]
has positive Hausdorff dimension for a certain interval of $\beta >0$ in the
case $l_n=\frac{\gamma}{n}$ with $\gamma>0$. For general $l_n$, Barral and
Fan \cite{BF} answered Carleson's problem by identifying three kinds of
phenomena depending whether the index
$\bar{\gamma}=\limsup_{N\to\infty}\frac{\sum_{n=1}^Nl_n}{-\log l_N}$ is zero,
positive and finite or infinite. More precisely, when $\bar\gamma=0$,
$\dimH F_\beta=1$ almost surely for all $\beta\ge 0$, when
$\bar{\gamma}=\infty$, $F_1=\mathbb T^1$ almost surely, and when
$0<\bar{\gamma}<\infty$, $\dimH F_\beta$ depends on $\beta$. Here the Hausdorff
dimension is denoted by $\dimH$.

For the case of zero Lebesgue measure, the Hausdorff dimension of $E$ was first
calculated by Fan and Wu~\cite{FanWu04} in the case $l_n=1 /n^\alpha$. When
studying the Hausdorff measure and the large intersection properties of $E$ for
general $l_n$, Durand~\cite{Durand10} gave another, independent proof of the
dimension result. According to \cite{FanWu04} and \cite{Durand10}, the
almost sure Hausdorff dimension of $E$ is given by
\begin{equation}\label{goal}
\dimH E=\inf \{t\geq 0\mid\sum_{n=1}^\infty l_n^t<\infty\}
=\limsup_{n\to\infty}\frac{\log n}{-\log l_n},
\end{equation}
where the lengths $l_n$ are in decreasing order. In \cite{Durand10}, the author
also proved that the packing dimension of $E$ equals $1$ almost surely.
When considering the hitting probability
property of the random set $E$, Li, Shieh and Xiao ~\cite{BingShiehXiao}
provided an alternative way to obtain the Hausdorff and packing dimension
results under some additional conditions. The result \eqref{goal} can be also
proven as a consequence of the mass transference principle due to Beresnevich
and Velani~\cite{BeresnevichVelani06} (see Proposition~\ref{easythm}). The fact that both packing and box counting dimensions are equal to $1$ almost surely follows since $E$ is almost
surely a dense $G_\delta$-set in $\mathbb T^1$(see ~\cite[Chapter 5, Proposition 11]{Kahane85} and \cite[Section 2]{Orponen}).

In this paper we study random covering sets in $d$-dimensional torus
$\mathbb T^d$. Letting $(g_n)$ be a sequence of subsets of $\mathbb T^d$
and letting $(\xi_n)$ be a sequence of independent random variables,
uniformly distributed on $\mathbb T^d$, define the random covering set by
\[
E=\limsup_{n\to\infty}(g_n + \xi_n)
=\bigcap_{n=1}^\infty\bigcup_{k=n}^\infty(g_k + \xi_k).
\]
Notice that a counterpart of \eqref{01 for E} is easily obtained, that is,
almost surely
\[
\mathcal L(E)=\begin{cases}
	0, \text{ when }\sum_{n=1}^\infty\mathcal L(g_n)<\infty\\
	1, \text{ when }\sum_{n=1}^\infty\mathcal L(g_n)=\infty
	\end{cases},
\]
where $\mathcal L$ is the Lebesgue measure on $\mathbb T^d$.

On the $d$-dimensional torus the Dvoretzky covering problem has been studied by
El H\'elou ~\cite{Helou78} and Kahane ~\cite{Kahane90} among others.
In \cite{Kahane90} Kahane gave a complete solution for the problem
when the sets $g_n$ are similar simplexes (see also Janson~\cite{Janson86}).
However, in the general case
the covering problem has not been completely solved.

For an overview on the research on random covering sets and related topics, we
refer to \cite[Chapter 11]{Kahane85}, the survey \cite{Kahane00} and the
references therein. Here we only mention a few variations on the classical
random covering model. For example, Hawkes~\cite{Hawkes73} considered under
which conditions all the points in $K\subset \mathbb T^1$ are covered with probability
one (or zero). Mandelbrot~\cite{Mandelbrot72}, in turn, introduced Poisson
covering of the real line (see also Shepp~\cite{Shepp72a}). In general metric
spaces, the random coverings by balls
have been studied by Hoffman-J\"orgensen~\cite{HoffmanJorgensen73}. Recent
contributions to the topic include various types of dynamical models, see Fan,
Schmeling and Troubetzkoy~\cite{FanSchmelingTroubetzkoy},
Jonasson and Steif~\cite{JonassonSteif08} and Liao and
Seuret~\cite{LiaoSeuret}.

We address the question of determining the analogue of \eqref{goal} in higher
dimensional case.
In ~\cite{FanWu04} the method is strongly adapted to the 1-dimensional case
whereas the argument based on the mass transference principle
~\cite{BeresnevichVelani06} can be carried through in any dimension
provided that the sets $g_n$ are uniformly ball like
(see Proposition~\ref{easythm}). Our main
interest is the case where the sets $g_n$ are not uniformly ball like, and
therefore, the mass transference principle cannot be applied.
It turns out that almost surely the Hausdorff dimension of
the covering set $E$ is given in terms of the singular value functions of the
linear mappings related to the system, see Theorem \ref{realthing}.

To this end, in Section \ref{lemmas} we introduce our setting, state our main
result and prove preliminary lemmas including the upper bound for the
dimension.
In Section \ref{construction} we construct a random subset of the covering set
$E$ having large dimension with positive probability which, in turn, gives
the lower bound of the dimension in Section \ref{dimension}.

\section{Preliminaries and statement of main theorem}
\label{lemmas}

Denote the closed ball of radius $r$ and centre $x$ in $\mathbb R^d$
by $B(x,r)$. Letting $L:\mathbb R^d\to\mathbb R^d$ be a contractive linear
injection, the image $L(B(0,1))$ of the unit ball $B(0,1)$ is an ellipse whose
semiaxes are non degenerated. The singular values
$0<\alpha_d(L)\leq \dots\leq \alpha_1(L)<1$ of $L$
are the lengths of the semiaxes of $L(B(0,1))$ in decreasing order.
Given $0< s\leq d$, define the {\it singular value function} by
\[
\Phi^s(L)=\alpha_1(L)\cdots\alpha_{m-1}(L)\alpha_m(L)^{s-(m - 1)},
\]
where $m$ is the integer such that $m-1<s\leq m$.

We use the notations $\mathbb T^d$ for the $d$-dimensional torus and
$\mathcal L$ for the Lebesgue measure on $\mathbb T^d$.
Consider a probability space $(\Omega, \mathcal A, P)$ and let
$(\xi_n)$ be a sequence of independent random variables which are uniformly
distributed
on $\mathbb T^d$, that is, $(\xi_n)_\ast P=\mathcal L$, where
$(\xi_n)_\ast P$ is the image measure of $P$ under $\xi_n$.
Letting $(g_n)$ be a sequence of subsets of $\mathbb T^d$, we use the notation
$G_n$ for the random translates $G_n=g_n+\xi_n\subset\mathbb T^d$
and define the random covering set generated by $(g_n)$ by
\[
E=E^\omega=\limsup_{n\to\infty}G_n.
\]

In this paper we consider the case $g_n=\Pi(L_n(R))$, where
$R\subset[0,1]^d$ has non-empty interior,
$L_n:\mathbb R^d\to\mathbb R^d$ is a contractive linear injection
for all $n\in\mathbb N$ and $\Pi:\mathbb R^d\to\mathbb T^d$ is the natural
covering map. Moreover, we assume that
for all $i=1,\dots ,d$ the sequence of singular values $\alpha_i(L_n)$
decreases to $0$ as $n$ tends to infinity. Defining
\begin{equation}\label{szero}
s_0=\inf\{0<s\le d\, |\, \sum_{n=1}^\infty \Phi^s(L_n)<\infty \},
\end{equation}
with the interpretation
$s_0=d$ if $\sum_{n=1}^\infty \Phi^d(L_n)=\infty$,
we are ready to state our main theorem.

\begin{theorem}\label{realthing} For $P$-almost all $\omega\in\Omega$ we have
\begin{equation}\label{dimnumber}
\dimH E^\omega=s_0.
\end{equation}
\end{theorem}

Theorem \ref{realthing} is an immediate consequence of the following
proposition concerning the case where each generating set $g_n$ is a rectangular
parallelepiped in $\mathbb T^d$ meaning that there exist a parallelepiped
$\tilde g_n\subset\mathbb R^d$ such that $g_n=\Pi(\tilde g_n)$.
In what follows rectangular parallelepipeds will consistently be called
rectangles.

Let $E(g_n)=E^\omega(g_n)$ be the covering set generated by a sequence
$(g_n)$ of rectangles.
For all rectangles $g$ and for all $0< s\leq d$ define
\[
\Phi^s(g)=\alpha_1(g)\cdots\alpha_{m-1}(g)\alpha_m(g)^{s-(m - 1)},
\]
where $0<\alpha_d(g)\leq \dots\leq \alpha_1(g)<1$ are the lengths the edges
of $g$ in decreasing order and $m$ is the integer such that $m-1<s\leq m$.

\begin{proposition}\label{almostrealthing} Assume that $(g_n)$ is a sequence
of rectangles such that for all $i=1,\dots,d$ the sequence of lengths
$\alpha_i(g_n)$ decreases to $0$ as $n$ tends to infinity. Then almost surely
\begin{equation}\label{dimnumberrec}
\dimH E(g_n)=s_0(g_n),
\end{equation}
where
\[
s_0(g_n)=\inf\{0<s\le d\, |\, \sum_{n=1}^\infty \Phi^s(g_n)<\infty \}
\]
with the interpretation $s_0=d$ if $\sum_{n=1}^\infty \Phi^d(g_n)=\infty$.
\end{proposition}

We proceed by verifying first that Theorem \ref{realthing} follows from
Proposition \ref{almostrealthing}.

\begin{proof}[Proof of Theorem \ref{realthing} as a consequence of
Proposition \ref{almostrealthing}]
Letting $(L_n)$, $R$ and $E$ be as in Theorem \ref{realthing}, there are
sequences
$(g'_n)$ and $(g_n)$ of rectangles such that
$g'_n\subset\Pi(L_n(R))\subset g_n$, and moreover,
$\alpha_i(g'_n)=c'\alpha_i(L_n)$ and
$\alpha_i(g_n)=c\alpha_i(L_n)$ for all $i=1,\dots,d$. Here the constants
$c'$ and $c$ are independent of $n$ and $i$. Since
$E(g'_n)\subset E\subset E(g_n)$ we have
\[
\dimH E(g'_n)\le\dimH E\le\dimH E(g_n).
\]
Applying Proposition \ref{almostrealthing} to the sequences $(g'_n)$
and $(g_n)$ and noting that $s_0(g'_n)=s_0(g_n)=s_0$, gives
\eqref{dimnumber}.
\end{proof}

It remains to prove Proposition \ref{almostrealthing}. As the first step we
verify the following lemma according to which the Hausdorff dimension of
$E(g_n)$ is always bounded above by $s_0(g_n)$. The proof is standard
following, for example, the ideas in ~\cite{Falconer88}.

\begin{lemma}\label{upper bound} Assume that $(g_n)$ and $s_0(g_n)$ are as
in Proposition \ref{almostrealthing}. Then
for all $\omega\in \Omega$ we have
$\dimH E^\omega(g_n)\leq s_0(g_n)$.
\end{lemma}

\begin{proof}
We may assume that $s_0(g_n)<d$. Let $s_0(g_n)<s<d$ and let $m$ be the
integer with
$m-1<s\leq m$. For each $n\in\mathbb N$ we estimate the
number of cubes of side length $\alpha_m(g_n)$ needed to cover $G_n$.
By expanding the last $d-m+1$ edges of $G_n$ to length $\alpha_m(g_n)$ and
by dividing the expanded rectangle to cubes of side length $\alpha_m(g_n)$,
we end up with an upper bound
\[
\left(\left\lfloor \frac{\alpha_1(g_n)}{\alpha_m(g_n)}\right\rfloor + 1\right )
\dots \left (\left\lfloor \frac{\alpha_{m-1}(g_n)}{\alpha_m(g_n)}\right\rfloor
+ 1\right )
\leq 2^{m-1}\alpha_1(g_n)\cdots\alpha_{m-1}(g_n)\alpha_m(g_n)^{-m+1},
\]
where the integer part of any $x\geq 0$ is denoted by $\lfloor x\rfloor$.

Recalling that for all $N\in\mathbb N$
\[
E(g_n)\subset \bigcup_{n=N}^\infty G_n,
\]
gives the following estimate for the $s$-dimensional Hausdorff measure
\[
\begin{aligned}
\mathcal H^s(E)&
\leq \liminf_{N\to\infty} \sum_{n=N}^\infty 2^{m-1}(\sqrt d \alpha_m(g_n))^s
\alpha_1(g_n)\cdots\alpha_{m-1}(g_n)\alpha_m(g_n)^{-m + 1} \\
&=\liminf_{N\to\infty}\sum_{n=N}^\infty  2^{m-1} (\sqrt d)^s \Phi^s(g_n) =0.
\end{aligned}
\]
This implies that $\dimH E(g_n)\leq s_0(g_n)$. \end{proof}

We continue by proving two auxiliary results.

\begin{lemma}\label{bigger than 1} Assume that $(L_n)$ is a sequence of
contractive linear injections $L_n:\mathbb R^d\to\mathbb R^d$.
Let $s_0$ be as in \eqref{szero} and let $m-1<s_0\leq m$. Defining for all
$m-1<s<s_0$
 \[
f(s):=\limsup_{n\to\infty}\frac{\log n}{-\log\Phi^s(L_n)},
\]
we have $f(s)>1$.
\end{lemma}

\begin{proof}
We will show that  $f(s)\geq 1$ for all $m-1<s<s_0$ and
$f$ is strictly decreasing. This clearly implies the claim.

Let $m-1<s<s_0$. The fact that
$\sum_{n=1}^\infty \Phi^s(L_n)=\infty$ implies that for all $\epsilon >0$
there exists a subsequence $(n_k)$ such that
$\Phi^s(L_{n_k})>\frac{1}{n_k^{1 + \epsilon}}$ for all $k$. From this we deduce
that $f(s)\ge\frac{1}{1+\epsilon}$, and letting $\epsilon$ go to $0$ yields
$f(s)\geq 1$.

Consider $\delta >0$ such that $m-1<s+\delta<s_0$.
Since $\Phi^s(L_n)\geq \alpha_m(L_n)^s$ we obtain
\[
\Phi^{s + \delta}(L_n)= \Phi^s(L_n)\alpha_m(L_n)^\delta
\leq \Phi^s(L_n)^{1 + \tfrac {\delta}{s}},
\]
giving
\[
f(s+\delta)\leq \limsup _{n\to\infty}
\frac{\log n}{(1 + \frac{\delta}{s})(-\log \Phi^s(L_n))}
=\frac{f(s)}{1 + \frac{\delta}{s}}<f(s).
\]
Hence $f$ is strictly decreasing.
\end{proof}

\begin{remark}
Lemma \ref{bigger than 1} holds for all $0<s<s_0$, but this stronger claim
is not necessary for our purposes.
\end{remark}

\begin{proposition}\label{probability estimate}
Assume that $G\subset \mathbb T^d$ and $\mathcal L(G)>0$. Let
$\xi_1,\dots,\xi_n$ be independent, uniformly distributed random variables on
$\mathbb T^d$.
Let
\[
M_n=\#\{i\in \{1,\dots ,n\}\mid\xi_i\in G\},
\]
where $\#A$ denotes the number of elements in a set $A$. Then
\[
P(M_n\leq \tfrac 12 n\mathcal L(G))
\leq \frac{4(1-\mathcal L(G))}{n\mathcal L(G)}.
\]
\end{proposition}

\begin{proof}
Denote by $\chi _A$ the characteristic function of a set $A$. Calculating
the first and second moments of $M_n$ gives
\[
\mathbb E(M_n)=\mathbb E(\sum_{i=1}^n \chi_{\{\xi_i\in G\}})=n\mathcal L(G)
\]
and
\[
\begin{aligned}
\mathbb E(M_n^2)&=\mathbb E\big(\big(\sum_{i=1}^n
\chi_{\{\xi_i\in G\}}\big)^2\big)
= \mathbb E\big(\sum_{i=1}^n\chi_{\{\xi_i\in G\}}
+ \sum_{j\neq i}\chi_{\{\xi_i\in G\}}\chi_{\{\xi_j\in G\}}\big)\\
&=n\mathcal L(G) + (n^2-n)\mathcal L(G)^2.
\end{aligned}
\]
From Chebyshev's inequality we deduce
\[
\begin{aligned}
P(M_n\leq \tfrac 12\mathbb E(M_n))&\leq P\big(|M_n-\mathbb E(M_n)|
\geq \tfrac 12\mathbb E(M_n)\big)\\
&\leq \frac{4(\mathbb E(M_n^2)-\mathbb E(M_n)^2)}{\mathbb E(M_n)^2}
=\frac{4(1-\mathcal L(G))}{n\mathcal L(G)}
\end{aligned}
\]
which completes the proof.
\end{proof}

\section{Construction of random Cantor sets}\label{construction}

Let $(g_n)$ and $s_0(g_n)$ be as in Proposition~\ref{almostrealthing}.
Consider an integer $m$ such that $m-1<s_0(g_n)\leq m$.
For notational simplicity, we assume that 0 is a vertex of each $g_n$. Indeed,
by choosing suitable deterministic translates, we find an isomorphic
probability space $(\Omega',\mathcal A',P')$ where this is the case since
the random variables $(\xi_n)$ are uniformly distributed and the rectangles
$(g_n)$ are deterministic.
For each $n$, let $T_n:\mathbb R^d\to\mathbb R^d$ be a linear map such that
$\Pi(T_n([0,1]^d))=g_n$. Observe that $\alpha_i(T_n)=\alpha_i(g_n)$ for
all $i=1,\dots,d$.
Let $m-1<s<s_0(g_n)$. For the purpose of proving Proposition
\ref{almostrealthing}
we construct in this section an event $\Omega(\infty)\subset \Omega$, having
positive probability,
and a random Cantor like set $C^\omega$ such that $C^\omega\subset E^\omega$
for all
$\omega\in \Omega(\infty)$.
In Section \ref{dimension} we prove that $\dimH C^\omega \geq s$ almost
surely conditioned on $\Omega(\infty)$. 

Let $a_0=\tfrac12$. Consider a sequence $(a_l)$ of real numbers larger than
1/2 increasing to $1$ with
$\Pi_{l=1}^\infty \tfrac 1{a_l} <\infty$. By Lemma \ref{bigger than 1}, there
exists a sequence $(n_k)$ of natural numbers satisfying
\begin{equation}\label{firstchoice}
\lim_{k\to\infty}\frac{\log n_k}{-\log\Phi^s(T_{n_k})}=f(s)>1.
\end{equation}
Moreover, by considering a suitable subsequence of $(n_k)$, we may assume that
for all $k\in\mathbb N$
\begin{align}
&\diam(g_{n_k})\le\frac 12(1-a_{k-1})\alpha_d(g_{n_{k-1}}),\label{diam}\\
&n_k\mathcal L(g_{n_{k-1}})\geq n_k^{\frac{3 + f(s)}{2 + 2f(s)}}\text{ and }
\label{exponent epsilon}\\
&\log n_k\ge n_{k-1}\label{nk large}
\end{align}
where $n_0=0$ and $g_0=\mathbb T^d$. Notice that since the sequence $(n_k)$
is deterministic it is independent of $\omega\in\Omega$.

We proceed by constructing inductively a random nested sequence
of finite collections $\mathcal C_k$ of rectangles as follows: Let
$\mathcal C_0=\{\mathbb T^d\}$ and $N_0=1$. Define
$N_1=\lfloor\tfrac12a_0^dn_1\rfloor$ and
$I(1,\mathbb T^d)= \{1,\dots, N_1\}$. For all $i\in I(1,\mathbb T^d)$,
let $g_i'$ be a linear isometric copy of $g_{n_1}$
contained in $g_i$. The existence of $g_i'$ follows from the fact that
$\alpha_j(g_{n_1})\leq \alpha_j(g_i)$ for all $i\leq n_1$ and
$j=1,\dots ,d$. For each $i\in I(1,\mathbb T^d)$, set $G_i'=g_i'+\xi_i$. Then
$G_i'\subset G_i$. Defining $\mathcal C_1=\{G_i'\,|\,i\in I(1,\mathbb T^d)\}$,
we have
\[
\bigcup _{G\in\mathcal C_1}G\subset \bigcup _{i=1}^{n_1}G_i.
\]
Furthermore, the collection $\mathcal C_1$ can be chosen for any
$\omega\in\Omega=:\Omega(1)$ giving $P(\Omega(1))=q_1$ with $q_1=1$.

Assume that there exist events $\Omega(1),\dots,\Omega(k-1)$ with
$P(\cap_{j=1}^{k-1}\Omega(j))=q_1\cdots q_{k-1}$ such that for all
$\omega\in\cap_{j=1}^{k-1}\Omega(j)$ there are
collections $\mathcal C_1,\dots ,\mathcal C_{k-1}$ having the following
properties for all $j=1,\dots,k-1$
\begin{align}
&(\tfrac 12)^{3}a_{j-1}^d(n_j-n_{j-1})\mathcal L(g_{n_{j-1}})\leq N_j\leq
  (n_j-n_{j-1})\mathcal L(g_{n_{j-1}}),
\text{ where }N_j=\#\mathcal C_j,\label{number of sets}\\
&\bigcup _{G\in\mathcal C_{j}}G\subset\bigcup_{G\in\mathcal C_{j-1}}G,
  \label{inclusion}\\
&\#\{G'\in \mathcal C_j\mid G'\subset G\}=\lfloor\tfrac 12 a_{j-1}^d m_j
  \mathcal L(g_{n_{j-1}})\rfloor
\text{ for each }G\in \mathcal C_{j-1} \label{divnumber sets} \\
&\text{ where }
m_j=\lfloor\tfrac{n_j-n_{j-1}}{N_{j-1}}\rfloor\notag\\
&\mathcal C_j\text{ is a finite collection of isometric copies of }g_{n_j}
  \text{ and }\label{rotations}\\
&\bigcup_{G\in\mathcal C_j}G\subset\bigcup_{l=n_{j-1} + 1}^{n_j}G_l.
  \label{inclusion again}
\end{align}
We define an event $\Omega(k)$ such that
$P(\cap_{j=1}^k\Omega(j))=q_1\cdots q_k$ and for all
$\omega\in\cap_{j=1}^k\Omega(j)$ there is a collection $\mathcal C_k$
satisfying \eqref{number of sets}--\eqref{inclusion again}.
Write
$\mathcal C_{k-1}=\{\widetilde G_1,\dots ,\widetilde G_{N_{k-1}}\}$
and set $m_k=\lfloor\tfrac{n_k-n_{k-1}}{N_{k-1}}\rfloor$.
For $l=1,\dots, N_{k-1}$, define random sets
\[
\widetilde I(k,\widetilde G_l)=\big\{i\in\{n_{k-1}+1+(l-1)m_k,\dots,n_{k-1}+lm_k
 \}\mid\xi_i\in a_{k-1}\widetilde G_l\big\},
\]
where $aG$ is the similar copy of $G$ with similarity ratio $a$ and with the
same centre as $G$. Let
\[
\Omega(k)=\big\{\omega\in\Omega\mid
  \#\widetilde I(k,G)> \tfrac 12 a_{k-1}^d m_k\mathcal L(g_{n_{k-1}})
\text{ for all }G\in\mathcal C_{k-1}\big\}
\]
and
\[
q_k=P(\Omega(k)|\Omega(1), \dots, \Omega(k-1)).
\]
Note that $q_k>0$. For each $G\in \mathcal C_{k-1}$ we denote by
$I(k,G)$ the collection of the first
$\lfloor\tfrac 12a_{k-1}^dm_k\mathcal L(g_{n_{k-1}})\rfloor$ elements in
$\widetilde I(k,G)$ and set
\[
\mathcal C_k=\{G_i'\mid G\in\mathcal C_{k-1},\, i\in I(k,G)\}\text{ and }
N_k=\#\mathcal C_k,
\]
where $G_i'=g_i'+\xi_i$ and $g_i'$ is a linear isometric copy of $g_{n_{k}}$
contained in $g_i$. (See Figure 1.)
\begin{figure}\label{figure}
\def\svgwidth{0.6\columnwidth}
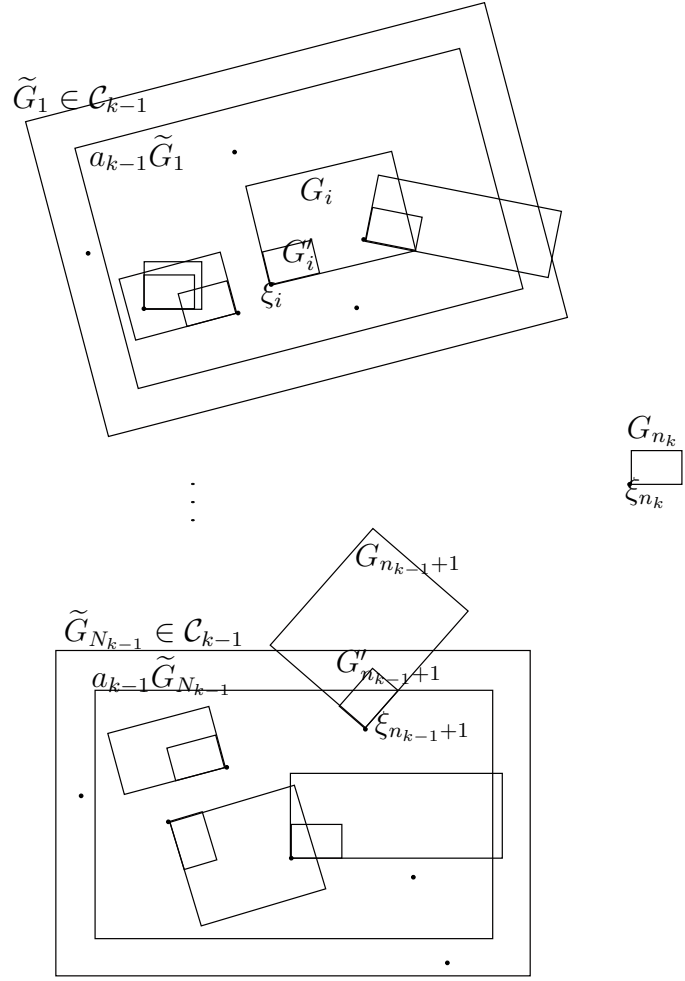
\caption{Construction of $\mathcal C_k$.}
\end{figure}
Observe that $N_k$ is deterministic.
As above, $g_i'$ exists since
$\alpha_j(g_{n_k})\le\alpha_j(g_i)$ for all $j=1,\dots,d$ and $i\leq n_k$.
Clearly, \eqref{divnumber sets} and \eqref{rotations} are valid for
$\mathcal C_k$. Since, by
inequality \eqref{diam}, we have $g'_i+\xi_i\subset G\in\mathcal C_{k-1}$
provided that
$\xi_i\in a_{k-1}G$, property \eqref{inclusion} holds for
$\mathcal C_k$. Furthermore, the choices of $m_k$ and $I(k,G_l)$ imply
\eqref{inclusion again}. The choice of $m_k$ gives
\begin{align*}
(\tfrac 12)^3a_{k-1}^d(n_k-n_{k-1})\mathcal L(g_{n_{k-1}})&\leq N_{k-1}
   \lfloor\tfrac 12 a_{k-1}^d m_k\mathcal L(g_{n_{k-1}})\rfloor=N_k\\
&\leq (n_k-n_{k-1})\mathcal L(g_{n_{k-1}}),
\end{align*}
and therefore,
condition (\ref{number of sets}) is satisfied for $\mathcal C_k$. Finally,
\[
P\big(\bigcap _{l=1}^k\Omega(l)\big)= P(\Omega(k)\mid\Omega(1),\dots,
  \Omega(k-1))P\big(\bigcap_{l=1}^{k-1}\Omega(l)\big)
=q_1\cdots q_k.
\]

Letting $\Omega(\infty)=\bigcap_{n=1}^\infty \Omega(n)$, we have
$P(\Omega(\infty))=\Pi_{n=1}^\infty q_n$. Define for all
$\omega\in\Omega(\infty)$
\[
C^\omega=\bigcap _{n=1}^\infty\bigcup_{G\in\mathcal C_n}G\subset E^\omega.
\]
Next we verify that the Cantor like set
$C^\omega\subset E^\omega$
exists with positive probability. We use the notation
$\mathcal F_k$ for the $\sigma$-algebra
generated by the random variables $\xi_1,\dots,\xi_{n_k}$.

\begin{proposition}\label{Nk deterministic} With the above notation we have
$P(\Omega(\infty))>0$.
\end{proposition}

\begin{proof}
We have
\begin{align*}
q_k&=P(\Omega(k)|\Omega(1),\dots,\Omega(k-1))\\
&=\frac{1}{P(\bigcap_{l=1}^{k-1}\Omega(l))} P\Big(\Omega(k)\cap\bigcap_{l=1}^{k-1}\Omega(l)\Big)\\
&=\frac{1}{P(\bigcap_{l=1}^{k-1}\Omega(l))}\mathbb E\Big(\mathbb E(\chi_{\Omega(k)}\chi_{\bigcap_{l=1}^{k-1}\Omega(l)}\mid \mathcal F_{k-1})\Big)\\
&=\frac{1}{P(\bigcap_{l=1}^{k-1}\Omega(l))}\mathbb E\Big(\mathbb \chi_{\bigcap_{l=1}^{k-1}\Omega(l)}\mathbb E(\chi_{\Omega(k)} \mid \mathcal F_{k-1})\Big)\\
&=\frac{1}{P(\bigcap_{l=1}^{k-1}\Omega(l))}\mathbb E\Big(\mathbb \chi_{\bigcap_{l=1}^{k-1}\Omega(l)}\mathbb E(\chi_{\bigcap\limits_{G\in\mathcal C_{k-1}}
\{\#\widetilde I(k,G)>\tfrac 12 a_{k-1}^d m_k \mathcal L(g_{n_{k-1}})\}} \mid \mathcal F_{k-1})\Big)\\
&\geq\frac{1}{P(\bigcap_{l=1}^{k-1}\Omega(l))}\mathbb E\Big(\mathbb \chi_{\bigcap_{l=1}^{k-1}\Omega(l)}\Big (1-\sum_{G\in \mathcal C_{k-1}}\mathbb E\big (\chi_{\{\#\widetilde I(k,G)\leq
  \tfrac 12 m_k \mathcal L(a_{k-1}g_{n_{k-1}})\}}\mid\mathcal F_{k-1}\big )\Big)\Big),
  \end{align*}
and applying Proposition \ref{probability estimate} hence gives 
\[
q_k\geq 1-N_{k-1}^2 \frac {8(1-\mathcal L(a_{k-1} g_{n_{k-1}}))}{(n_k-n_{k-1})
  \mathcal L(a_{k-1} g_{n_{k-1}})}=:1-p_k.
\]
Inequalities \eqref{number of sets} and \eqref{nk large}, in turn, imply that
$N_{k-1}\leq (n_{k-1}-n_{k-2}) \mathcal L(g_{n_{k-2}})\leq n_{k-1}\leq \log n_k$,
and therefore, noting that $n_k-n_{k-1}\geq\tfrac 12 n_k$ by \eqref{nk large}
and using \eqref{exponent epsilon}, we obtain
\[
\sum_{k=1}^\infty p_k\le\sum_{k=1}^\infty\frac{8(1-\mathcal L(a_{k-1}g_{n_{k-1}}))
  (\log n_k)^2}{\tfrac 12a_{k-1}^dn_k\mathcal L(g_{n_{k-1}})}
\leq \sum_{k=1}^\infty \frac{8(\log n_k)^2}{\tfrac 12a_{k-1}^d n_k^{(3+f(s))/
  (2+2f(s))}}<\infty,
\]
where the convergence follows since by \eqref{nk large} the sequence $(n_k)$ is
growing exponentially fast. Letting $k_0\in\mathbb N$ be such that $p_k<1$ for
all $k\ge k_0$, we have
$\Pi_{k=1}^\infty q_k\geq\Pi_{k=1}^{k_0}q_k\Pi_{k=k_0+1}^\infty(1-p_k)>0$.
\end{proof}

\begin{remark} 
The idea of finding a large-dimensional Cantor subset of the random covering set was already exploited in the dimension calculation of Fan and Wu \cite{FanWu04} in the case of $\mathbb T^1$. In their proof it is essential that the sets $C^\omega$ are homogeneous and the construction intervals are well-separated, which follows from well-known results on random spacings of uniform random samples \cite{Hawkes85}. Structure of the set allows them then to directly estimate sizes of intersections of balls with the set $C^\omega$, giving the dimension bound from below. In our choice of the subset $C^\omega$, however, separation of the generating sets plays no role. Indeed, it is a well-known fact that for self-affine sets no separation condition guarantees the dimension formula. Also a direct estimate for measures of balls is probably hopeless. Instead a potential theoretic method based on a transversality argument is the key, see lemma \ref{falconer} below. In the implementation of this idea we need the assumption \eqref{rotations}. 
\end{remark}

\section{Dimension estimate}\label{dimension}

Using the notation introduced in
Section ~\ref{construction}, we prove that for $s<s_0(g_n)$ the event
$\{\omega\in\Omega(\infty)\mid\dimH C^\omega\geq s\}$ has positive probability.
To obtain the dimension bound, we use potential theoretic methods and
define a measure supported on $C^\omega$ with finite $s$-energy. In what follows, we consider only the event $\Omega(\infty)$ and denote the
expectation over $\Omega(\infty)$ simply by $\mathbb E$.

For any $\omega\in\Omega(\infty)$, $k\in\mathbb N$ and $G\in\mathcal C_{k-1}$,
let $M_k=\#I(k,G)=\lfloor\tfrac 12a_{k-1}^dm_k\mathcal L(g_{n_{k-1}})\rfloor$
be the number of level $k$ construction rectangles contained in
$G$. Notice that $M_k$ is a deterministic number depending only on $k$. For
later notational simplicity, we will relabel the random variables
$\xi_i$ using a deterministic tree structure.

For all $l\in\mathbb N$, consider the sets
$\J_l=\{i_1\dots i_l\mid i_j\in\{1,\dots,M_j\}
\text{ for all }j\in\{1,\dots,l\}\}$ and define
$\J=\bigcup _{l=0}^\infty\J_l$,
 with the convention
$\J_0=\{\emptyset\}$. For $\bi$, $\bj\in \J$, denote by $\bi\wedge\bj$ the
maximal common initial sequence of $\bi$ and $\bj$ and let $\bi\bj\in\J$ be
the word obtained by juxtaposing the words $\bi$ and $\bj$.
Further, we denote by $|\bi|$ the length of $\bi\in\J$, that is, $|\bi|=l$ if
$\bi\in\J_l$. For each $l\leq k$ and $\bi\in \J_l$, define the
{\it cylinder} of length $l$ and of depth $k$ by $C(\bi,k)=\{\bj\in \J_k\,|\,\bi\wedge\bj=\bi\}.$ For $i\in\{1,\dots,M_1\}$, define $\phi_i=\xi_i$ and $G(i)=g_i'+\phi_i$ and let
$T_i'$ be a linear map such that $\Pi(T_i'([0,1]^d)=g_i'$.
Assume that we have defined the random variables $\phi_\bi$ and the rectangles
$G(\bi)\in\mathcal C_{k-1}$ for all $\bi\in \J_{k-1}$. Let
$I(k,G(\bi))=\{j_1,\dots, j_{M_k}\}$ where $j_i<j_{i+1}$ in the natural order
given by the construction. For all $i\in \{1,\dots, M_k\}$, define
$\phi_{\bi i}=\xi_{j_i}$, $g_{\bi i}'=g_{j_i}'$ and
$G(\bi i)=g_{\bi i}'+\phi_{\bi i}$ and let $T_{\bi i}'$ be a linear map
satisfying $\Pi(T'_{\bi i}([0,1]^d))=g_{\bi i}'$.
Then $\det(T_{n_{|\bi|}})=\mathcal L(G(\bi))$
and $\Phi^s(T_{\bi}')=\Phi^s(T_{n_{|\bi|}})$ for all $\bi\in\J$.
For notational purposes set $G(\emptyset)=\mathbb T^d$ and $\phi_\emptyset=0$.
When necessary we
view $T'_{\bi}$ as a map on $\mathbb T^d$ by identifying $\mathbb T^d$ with
$[0,1[^d$. Finally, for $\bi_1,\dots,\bi_k\in\J$, denote by $\mathcal F(\bi_1,\dots,\bi_k)$ the
$\sigma$-algebra generated by the events
$\{\omega\in\Omega(\infty)\mid G(\bi_l)=Q_l\text{ for all }l=1,\dots,k\}$,
where each $Q_l\subset\mathbb T^d$ is an isometric copy of $g_{n_{|\bi_l|}}$.
\begin{remark}\label{cond independence}
Note that $\{\phi_{\bi}\mid\bi\in C(\bj,k)\}=\{\xi_i\mid i\in I(k,G(\bj))\}$
for
any $\bj\in \J_{k-1}$ and
$\{\phi_{\bi}\mid\bi\in\J_k\}=\{\xi_i\mid i\in\bigcup_{G\in\mathcal C_{k-1}}I(k,G)
   \}$. Let $A\subset \mathbb T^d$ be a Borel set with $\mathcal L(A)>0$. Since $\xi_j$ is uniformly
distributed on $\mathbb T^d$ for given $j$, every $\xi_j$ is uniformly distributed 
on $A$ when conditioned on the event $\xi_j\in A$. Let $i\in\mathbb N$ and let $\bi 
i\in\J_{k+1}$. By definition $\phi_{\bi i}=\xi_j$ for some $j\in\{n_k+1,
\dots,n_{k+1}\}$ with $\xi_j\in G(\bi)$, and hence the random variable
$\phi_{\bi i}$ is uniformly distributed on
$a_kG(\bi)$ when conditioned on $\phi_{\bi i}=\xi_j$ and the $\sigma$-algebra $\mathcal F(\bi)$. Furthermore, 
\begin{align*}
\mathbb E(\chi_{\{\phi_{\bi i}\in A\}}&\mid \mathcal F(\bi))=\sum_{j=n_{k}+1}^{n_{k+1}}\mathbb E(\chi_{\{\phi_{\bi i}\in A\}}\mid \mathcal F(\bi), \phi_{\bi i}=\xi_j)\mathbb E(\chi_{\{\phi_{\bi i}=\xi_j\}}\mid \mathcal F(\bi))\\
&= \frac{\mathcal L(A\cap a_kG(\bi))}{\mathcal  L(a_kG(\bi))}\sum_{j=n_{k}+1}^{n_{k+1}} \mathbb E(\chi_{\{\phi_{\bi i}=\xi_j\}}\mid \mathcal F(\bi))=\frac{\mathcal L(A\cap a_kG(\bi))}{\mathcal L(a_kG(\bi))}.
\end{align*}
Hence $\phi_{\bi i}$ is uniformly distributed inside $a_k G(\bi)$ when conditioned on $\mathcal F(\bi)$. Moreover, if $\bj$ satisfies $\bj\wedge\bi i\ne\bi i$, conditioning on $\mathcal F(\bi, \bj)$ instead of $\mathcal F(\bi)$ does not change the
uniform distribution of $\phi_{\bi i}$ on $a_kG(\bi)$, since $\xi_j$ and $\xi_l$ are independent for
$j\ne l$. Recall that even though the corner
points $\phi_{\bi i}$ and $\phi_{\bi h}$ are independent for $i\ne h$, the
rectangles $G(\bi i)$ and $G(\bi h)$ are not, since the orientation of
$g'_{\bi i}$ is determined by the index $j_i$.
\end{remark}

\begin{lemma}\label{spt}
The sequence of measures $\mu^\omega_l$ on $\mathbb T^d$ given by
\begin{equation}\label{mul}
\mu^\omega_l=\frac{\sum_{\bi\in \J_l}
  (T_{\bi}'+\phi_\bi)_*\mathcal L}{ N_l}
\end{equation}
converges in weak$^*$-topology to a measure $\mu^\omega$ supported on $C^\omega$. 
\end{lemma}
\begin{proof}
By the Riesz representation theorem a weak$^*$-limit $\mu^\omega $ exists, if we prove that for all positive, continuous functions $f$ on $\mathbb T^d$ the sequence $\int f\,d\mu^\omega_l$ converges. 

To that end, fix a positive, continuous function $f$ on $\mathbb T^d$ and $\epsilon >0$. Since $\mathbb T^d$ is compact, there exists $\delta>0$ with $|f(x)-f(y)|<\epsilon$ for all $|x-y|<\delta$. Let $K$ be so large that $\diam(g_{n_K})<\delta$, and fix $k\geq K$. Write $\mu^\omega_k$ as a sum of measures $\mu_{\bi,k}^\omega$ defined by
\[
\mu_k^\omega=\sum_{\bi\in\J_K}\sum_{\bj\in C(\bi, k)}\frac{(T_\bj' + \phi_\bj)_*\mathcal L}{N_k} = \sum_{\bi\in\J_K}\mu_{\bi,k}^\omega.
\]
For all $\bi\in\J_K$, we have $\mu_{\bi,k}^\omega(G(\bi))=\tfrac 1{N_K}=\mu_{\bi,K}^\omega(G(\bi))$ and $\spt\mu_{\bi, k}^\omega\subset G(\bi)$. Therefore,
\[
|\int f\, d\mu^\omega_{k} - \int f \,d\mu^\omega_{K}|\leq\sum_{\bi\in\J_K}|\int_{G(\bi)} f\, d\mu^\omega_{\bi, k} - \int_{G(\bi)} f \,d\mu^\omega_{\bi, K}|\leq \epsilon, 
\]
since $\diam G(\bi)=\diam (g_{n_K})<\delta$. Thus sequence $\int f \,d\mu_{l}^\omega$ converges. The claim $\spt \mu^\omega\subset C^\omega$ holds since $C^\omega$ is compact and $\spt\mu^\omega_l\subset \cup_{G\in\mathcal C_l}G$ for all $l$. 
\end{proof}

Next we show that for all $s<s_0(g_n)$ the $s$-energy
$I^s(\mu^\omega)=\iint\frac{d\,\mu^\omega(x) d\,\mu^\omega(y)}{\vert x-y\vert^s}$
of $\mu^\omega$ is finite almost surely.
In the energy estimate we will make use of the following lemma
\cite[Lemma 2.2]{Falconer88}.

\begin{lemma}[Falconer]\label{falconer}
Let $s$ be non-integral with $0<s<d$ and let
$T:\mathbb R^d\to\mathbb R^d$ be an
affine injection. Then there exists a number $0<D_0<\infty$, depending only on
$d$ and $s$, such that
\[
\int_{[0,1]^d}\frac{d\mathcal L(x)}{|T(x)|^s}\leq \frac {D_0}{\Phi^s(T)}.
\]
\end{lemma}

\begin{lemma}\label{** cond}
For all $\bi,\bj\in\J$ and $x,y\in\mathbb T^d$ we have
\[
\mathbb E\big(\chi_{G(\bj)}(y)\chi_{G(\bi)}(x)\big)
\leq\big(\Pi_{l=1}^\infty\tfrac 1{a_l} \big)^{2d}\frac{\det(T_{n_{|\bi|}})
  \det(T_{n_{|\bj|}})}{\det(T_{n_{|\bi\wedge\bj|}})^2}\mathbb E\big(
  \chi_{G(\bi \wedge\bj)}(y)\chi_{G(\bi\wedge\bj)}(x)\big).
\]
\end{lemma}

\begin{proof}
Since $\left(\Pi_{l=1}^\infty \tfrac 1{a_l}\right)>1$, the claim holds when
$\bi=\bj$. Consider $\bi\neq\bj\in \J$. Without loss of generality, we may
assume that
$|\bi|\geq |\bj|$. Letting $\bk\in\J$ and $i\in\mathbb N$ satisfy
$\bi =\bk i\in \J$,
we obtain for any $x,y\in \mathbb T^d$ that
\begin{align*}
\mathbb E\big(\chi_{G(\bj)}(y)\chi_{G(\bi)}(x)\big)
&=\mathbb E\big(\chi_{G(\bj)}(y)
  \chi_{G(\bk)}(x)\chi_{G(\bi)}(x)\big)\\
&=\mathbb E\big( \chi_{G(\bj)}(y)\chi_{G(\bk)}(x)\mathbb E(\chi_{G(\bi)}(x)
  \mid\mathcal F(\bj,\bk))\big)\\
&=\mathbb E\big( \chi_{G(\bj)}(y)\chi_{G(\bk)}(x)\mathbb E(\chi_{x-g_{\bi}'}
  (\phi_{\bi})\mid\mathcal F(\bj,\bk))\big).
\end{align*}
Even though the orientation of $g_\bi'$ depends on $\omega\in\Omega(\infty)$,
the volume $\mathcal L(g_\bi')$ does not. Therefore, from
Remark~\ref{cond independence} we get
\[
\mathbb E\big(\chi_{x-g_{\bi}'}(\phi_{\bi})\mid\mathcal F(\bj,\bk)\big)
  \le\frac{\mathcal L(g_{n_{|\bi|}})}{\mathcal L(a_{|\bk|}g_{n_{|\bk|}})},
\]
and therefore,
\begin{align*}
\mathbb E\big(\chi_{G(\bj)}(y)\chi_{G(\bi)}(x)\big)
&\le\mathbb E\Bigl(\chi_{G(\bj)}(y)\chi_{G(\bk)}(x)
  \frac{\mathcal L(g_{n_{|\bi|}})}{\mathcal L(a_{|\bk|}g_{n_{|\bk|}})}\Bigr)\\
&=\frac{\det(T_{n_{|\bi| }})}{a_{|\bk|}^d\det(T_{n_{|\bk|}})}\mathbb E\big(
  \chi_{G(\bj)}(y)\chi_{G(\bk)}(x)\big).
\end{align*}
Iterating this with respect to $\bk$, if necessary, gives
\begin{equation}\label{chis}
\mathbb E\big(\chi_{G(\bj)}(y)\chi_{G(\bi)}(x)\big)\leq\big(\Pi_{l=1}^\infty
  \tfrac 1{a_l}\big)^d\frac{\det(T_{n_{|\bi|}})}{\det(T_{n_{|\bi\wedge\bj|}})}
  \mathbb E\big(\chi_{G(\bj)}(y)\chi_{G(\bi\wedge\bj)}(x)\big).
\end{equation}
Inequality \eqref{chis} completes the proof provided that
$\bj=\bi\wedge\bj$. If this is not the case,
we apply the above argument with $\bj$ playing the role of $\bi$ and
$\bi\wedge\bj$ playing that of $\bj$.
\end{proof}

Lemmas ~\ref{falconer} and ~\ref{** cond} lead to the following
energy estimate.

\begin{proposition}\label{energy uniform} Letting $s<s_0(g_n)$,
there exists a constant $C<\infty$ such
that $\int_{\Omega(\infty)}I^s(\mu^\omega_l )\,dP(\omega)<C$ for all
$l\in\mathbb N$. In particular, $I^s(\mu^\omega)<\infty$ for $P$-almost all
$\omega\in\Omega(\infty)$.
\end{proposition}

\begin{proof} Let $s<s_0(g_n)$ and let
$\bi,\bj\in\J$. Define
\[
H(\bi,\bj,s)=\int_{\mathbb T^d}\int_{\mathbb T^d}\frac{1}{|x-y|^s}\,
  d(T_{\bi}'+\phi_\bi)_*\mathcal L(x)\,d(T_{\bj}'+\phi_\bj)_*\mathcal L(y).
\]
As the functions involved are clearly measurable, use of Fubini's theorem and Lemmas \ref{** cond} and \ref{falconer} yields
the following estimate
\begin{equation*}
\begin{aligned}
\int_{\Omega(\infty)} H(\bi,\bj,s)\,dP
&=(\det T_{n_{|\bi|}}\det T_{n_{|\bj|}})^{-1}\int_{\mathbb T^d}\int_{\mathbb T^d}
  \frac{\mathbb E(\chi_{G(\bi)}(x)\chi_{G(\bj)}(y))}{|x-y|^s}\,d\mathcal L(x)\,
  d\mathcal L(y)\\
&\leq\left(\Pi_{l=1}^\infty \tfrac 1{a_l}\right)^{2d}
\int_{\mathbb T^d}\int_{\mathbb T^d}\frac{\mathbb E(\chi_{G(\bi\wedge\bj)}(x)
  \chi_{G(\bi\wedge\bj)}(y))}{\det(T_{\bi\wedge\bj}')^2|x-y|^{s}}\,d\mathcal L(x)
  \,d\mathcal L(y)\label{different}\\
&=\left( \Pi_{l=1}^\infty \tfrac 1{a_l} \right)^{2d}\int_{\Omega(\infty)}
  \int_{\mathbb T^d}\int_{\mathbb T^d}\frac{d\mathcal L(x)\,d\mathcal L(y)\,dP}
  {|T'_{\bi\wedge\bj}(x-y)|^s}\leq\frac{D}{\Phi^s(T_{n_{|\bi\wedge\bj|}})},
\end{aligned}
\end{equation*}
where $D$ depends on $D_0$ of Lemma \ref{falconer}. Combining this with \eqref{mul} gives
\begin{align*}
\int_{\Omega(\infty)}I^s(\mu_l^\omega)\,dP(\omega)&=\int_{\Omega(\infty)}
  \int_{\mathbb T^d}\int _{\mathbb T^d}\frac{1}{|x-y|^s}\,d\mu_l^\omega(x)\,
  d\mu_l^\omega(y)\,dP(\omega)\\
&=\frac{\sum_{\bi\in \J_l}\sum_{\bj\in\J_l}}{N_l^2}\int H(\bi,\bj,s)\,dP\leq N_l^{-2}\sum_{\bi\in J_l}\sum_{\bj\in\J_l}\frac{D}{\Phi^s(T_{n_{|\bi\wedge\bj|}})}\\
&\leq N_l^{-2}\sum_{K=0}^l\sum_{\bk\in\J_K}\sum_{\bi\in C(\bk,l)}\sum_{\bj\in C(\bk,l)}\frac{D}{\Phi^s(T_{n_K})}=\sum_{K=0}^l \frac {D}{N_k\Phi^s(T_{n_K})}.
\end{align*}
From \eqref{firstchoice} we deduce that
$\Phi^s(T_{n_k})>n_k^{\frac{-2}{1 + f(s)}}$ for
large $k$. Recalling \eqref{number of sets}, \eqref{nk large} and
\eqref{exponent epsilon}, gives for large $k$ that
\begin{equation}\label{sum big}
\begin{aligned}
N_k\Phi^s(T_{n_k})&\geq(\tfrac 12)^{3}a_{k-1}^d(n_k-n_{k-1})\mathcal L(g_{n_{k-1}})
  \Phi^s(T_{n_k})\\
&\geq (\tfrac 12)^{4}a_{k-1}^dn_k\mathcal L(g_{n_{k-1}})\Phi^s(T_{n_k})
\geq (\tfrac 12)^{4}a_{k-1}^dn_k^{\frac{3+f(s)}{2+2f(s)}}n_k^{\frac{-2}{1 + f(s)}}\\
&=(\tfrac 12)^{4}a_{k-1}^dn_k^{\frac{f(s)-1}{2+2f(s)}}.
\end{aligned}
\end{equation}
By \eqref{nk large} the sequence $(n_k)$ is growing exponentially fast. Therefore, recalling that $f(s)-1<0$, inequality \eqref{sum big} implies that the series
$\sum_{K=0}^\infty\frac{D}{N_K\Phi^s(T_{n_K})}$
converges.
The final claim follows by approximating the kernel $|x|^{-s}$ by kernels
$\min\{|x|^{-s},A\}$, where $A\in\mathbb N$.
\end{proof}

Now Proposition~\ref{almostrealthing} follows in a straightforward manner.

\begin{proof}[Proof of Proposition \ref{almostrealthing}]
By Lemma \ref{upper bound} it suffices to prove that
$\dimH E\geq s_0(g_n)$.
Consider $m-1<s<s_0(g_n)\leq m$ where $m$ is an integer.
Lemma~\ref{spt} and Proposition~\ref{energy uniform} combined with
\cite[Theorem 8.7]{Mattila} imply that
$\dimH C^\omega \geq s$ almost surely conditioned on $\Omega(\infty)$ which,
in turn, gives
\[
P(\dimH E^\omega\geq s)>0.
\]
Since $\{\dimH E\geq s\}$ is a tail event, from the Kolmogorov zero-one law
we deduce that
$P(\dimH E\geq s)=1$.
Approaching $s_0(g_n)$ along an increasing sequence of real numbers
$s$ gives
$\dimH E^\omega\geq s_0(g_n)$ for $P$-almost all $\omega\in\Omega$.
\end{proof}

As we mentioned in the introduction, for ball like covering sets
the dimension formula is an easy consequence of the mass transference
principle of Beresnevich and Velani. Since the proof is quite simple in this
case, we give the details here.

For a ball $B=B(x,r)\subset\mathbb R^d$ and $0<s<d$, write
$B^s=B(x,r^{\frac{s}{d}})$.
We recall a special
case of the mass transference principle
\cite[Theorem 2]{BeresnevichVelani06} suitable for our purposes.

\begin{theorem}[Beresnevich-Velani]\label{MTP}
Let $(B_n)\subset\mathbb R^d$ be a sequence of balls whose
radii converge to zero. Suppose that for any ball $B\subset\mathbb R^d$
\[
\mathcal H^d(B\cap\limsup_{n\to\infty}B_n^s)=\mathcal H^d(B).
\]
Then for any ball $B$ in $\mathbb R^d$,
\[
\mathcal H^s(B\cap\limsup_{n\to\infty} B_n)=\infty.
\]
\end{theorem}

\begin{proposition}\label{easythm}
Consider a sequence $(g_n)$ of subsets of $\mathbb T^d$ satisfying
$B(x_n,r_n)\subset g_n$ for sequences of points $(x_n)$ and
radii $(r_n)$. Letting $\rho_n$ be the diameter of $g_n$ with
$\rho_n\downarrow 0$, assume that
there exists $C<\infty$ such that $\frac{\rho_n}{r_n}\le C$ for all
$n\in\mathbb N$. Let $(\xi_n)$ be a sequence of independent random variables, uniformly distributed on $\mathbb T^d$. Then for $E=\limsup_{n\to\infty} (g_n + \xi_n)$, almost surely
\[
\dimH E=\min\{s_0,d\},
\]
where $s_0=\inf\{s\geq 0\mid\sum_{n=1}^\infty\rho_n^s<\infty\}$.
\end{proposition}

\begin{proof}
Let $s>s_0$. Set $G_n=g_n+\xi_n$. Since $E\subset\bigcup_{n=N}^\infty G_n$ for all $N$, we obtain
\[
\mathcal H^s(E)\leq\liminf_{N\to\infty}\sum_{n=N}^\infty\rho_n^s =0,
\]
giving $\dimH E\leq\min\{s_0,d\}$.

Obviously, $E\supset\limsup_{n\to\infty}B_n$ where $B_n=B(x_n + \xi_n,r_n)$.
Consider $s<\min\{s_0,d\}$.
Letting $K=\mathcal L(B(0,1))$, we have
\begin{equation}\label{BC}
\sum_{n=1}^\infty \mathcal L(B_n^s)= K\sum_{n=1}^\infty r_n^s\geq
  KC^{-s}\sum_{n=1}^\infty\rho_n^s=\infty.
\end{equation}
Since $P(x\in B_n^s)=\mathcal L(B_n^s)$ for all $x\in\mathbb T^d$ and $n\in\mathbb N$, Borel-Cantelli lemma and  \eqref{BC} imply $P(x\in \limsup_{n\to\infty} B_n^s)=1$. Applying Fubini's theorem, gives $\mathcal L(\limsup_{n\to\infty}B_n^s)=1$ almost surely, implying
$\mathcal L(\limsup_{n\to\infty}B^s_n\cap B)=\mathcal L(B)$ for any ball
$B\subset \mathbb T^d$. From Theorem~\ref{MTP} we get
$\mathcal H^s(\limsup_{n\to\infty} B_n)=\infty$, which leads to
$\dimH E\ge\min\{s_0,d\}$, almost surely. \end{proof}

\begin{remark}\label{alldecrease}
In $\mathbb T^1$ one may assume without loss of generality that $(l_n)$ is a
decreasing sequence by reordering the sequence if necessary whereas in
$\mathbb T^d$ with $d>1$ one cannot always reorder
$\alpha_i(L_n)$ simultaneously for all $i=1,\dots,d$. However, we do not
know whether this assumption is necessary for the validity of
Theorem~\ref{realthing}.
\end{remark}

\end{document}